\documentclass[a4paper,10pt]{amsart}
\usepackage{amsmath,amssymb,amsthm,amscd}
\usepackage[all]{xy}
\renewcommand{\iff}{if and only if }
\newcommand{\st}{such that }

\newcommand{\RMod}{\hbox{R{\rm -Mod}}}
\newcommand{\ModR}{\hbox{{\rm Mod-}R}}

\newcommand{\Pot}{\mathfrak{P}}

\newcommand{\Z}{\mathbb{Z}}
\newcommand{\Q}{\mathbb{Q}}
\DeclareMathOperator{\Hom}{Hom}

\DeclareMathOperator{\Ext}{Ext}
\DeclareMathOperator{\Tor}{Tor}
\DeclareMathOperator{\Ker}{Ker}
\DeclareMathOperator{\Img}{Im}

\usepackage[usenames]{color}

\theoremstyle{plain}
\newtheorem{thm}{Theorem}[section]
\newtheorem{prop}[thm]{Proposition}
\newtheorem{lem}[thm]{Lemma}

\theoremstyle{definition}
\newtheorem{defn}[thm]{Definition}
\newtheorem{exm}[thm]{Example}

\theoremstyle{remark}
\newtheorem{rem}{Remark}

\begin{document}
\title[Approximations, Mittag-Leffler conditions -- tools]{Approximations and Mittag-Leffler conditions \qquad the tools}

\author{\textsc{Jan \v Saroch}}
\address{Charles University, Faculty of Mathematics and Physics, Department of Algebra \\ 
Sokolovsk\'{a} 83, 186 75 Praha~8, Czech Republic}
\email{saroch@karlin.mff.cuni.cz}
 

\thanks{Research supported by grant GA\v CR 14-15479S}

\date{\today}

\begin{abstract} Mittag-Leffler modules occur naturally in algebra, algebraic geometry, and model theory, \cite{RG}, \cite{Gr}, \cite{P2}. If $R$ is a non-right perfect ring, then it is known that in contrast with the classes of all projective and flat modules, the class of all flat Mittag-Leffler modules is not deconstructible \cite{HT}, and it does not provide for approximations when $R$ has cardinality $\leq \aleph_0$, \cite{BS}. We remove the cardinality restriction on $R$ in the latter result. We also prove an extension of the Countable Telescope Conjecture \cite{SS}: a cotorsion pair $(\mathcal A,\mathcal B)$ is of countable type whenever the class $\mathcal B$ is closed under direct limits. 

In order to prove these results, we develop new general tools combining relative Mittag-Leffler conditions with set-theoretic homological algebra. They make it possible to trace the facts above to their ultimate, countable, origins in the properties of Bass modules. These tools have already found a number of applications: e.g., they yield a positive answer to Enochs' problem on module approximations for classes of modules associated with tilting \cite{AST2}, and enable investigation of new classes of flat modules occurring in algebraic geometry \cite{SlT2}. Finally, the ideas from Section~3 have led to the solution of
a long-standing problem due to Auslander on the existence of right almost split maps~\cite{S}.
\end{abstract} 

\maketitle
\vspace{4ex}

\section{Introduction}
\label{sec:intro}

The Mittag-Leffler condition for countable inverse systems was introduced by Grothendieck as a sufficient condition for the exactness of the inverse limit functor,~\cite{Gr}. Flat Mittag-Leffler modules then played an important role in the proof of the locality of the notion of a vector bundle in \cite[Seconde partie]{RG}. The renewed interest in Mittag-Leffler conditions stems from their role in the study of roots of the Ext functor, see \cite{ABH}, \cite{AH}, \cite{BH}, and \cite{SS}, as well as their connection to generalized vector bundles \cite{D}, \cite{EGPT}. In \cite{HT}, it was shown that flat Mittag-Leffler modules coincide with the $\aleph_1$-projective modules, studied already by Shelah et al.\ by  means of set-theoretic homological algebra, \cite{EM}. A connection to model theory was discovered in \cite{R}. The many facets of the notion of a Mittag-Leffler module, as well as its many applications, make it a key notion of contemporary algebra.      
 
Though locally projective, flat Mittag-Leffler modules have a very complex global structure in the case when $R$ is not right perfect. Similarly to the classes $\mathcal P_0$ and $\mathcal F_0$ of all projective and flat modules, the class $\mathcal F \mathcal M$ of all flat Mittag-Leffler modules is closed under transfinite extensions. However, unlike $\mathcal P_0$ and $\mathcal F_0$, it is not possible to obtain $\mathcal F \mathcal M$ by transfinite extensions from any set of its elements. In other words, $\mathcal F \mathcal M$ is not deconstructible, \cite{HT}. Moreover, unlike $\mathcal P_0$ and $\mathcal F_0$, the class $\mathcal F \mathcal M$ does not provide for approximations when $R$ is countable, \cite{BS}. Further results of this kind have been proved in \cite{SlT}. All these results, however, have the drawback of assuming that either the ring or the modules satisfy restrictive conditions on their size or structure. 

\medskip
Our goal here is to remove these restrictions. We succeed in doing so in the case of Mittag-Leffler modules by proving that $\mathcal F \mathcal M$ is not precovering for an arbitrary non-right perfect ring (Theorem \ref{t:nonprec}). 

However, our key tool, Lemma \ref{l:nonprec}, is much more general. It makes it possible to test for non-existence of approximations using certain small, countably presented modules, called the Bass modules. The terminology comes from the prototype example of a Bass module, namely the direct limit of the countable system  
$$R \overset{f_0}\to R \overset{f_1}\to \dots \overset{f_{n-1}}\to R \overset{f_n}\to R \overset{f_{n+1}}\to \dots$$  
where $f_n$ is the left multiplication by $a_n$ in $R$, and $Ra_0 \supseteq Ra_1a_0 \supseteq \dots$ is a decreasing chain of principal left ideals in $R$. By a classic result of Bass, the direct limit is projective, if and only if the chain stabilizes. So non-right perfect rings are characterized by the existence of such Bass modules that are not projective. Recently, Lemma \ref{l:nonprec} has been applied to a number of diverse settings: e.g., to locally $T$-free modules in \cite{AST2}, and locally very flat modules in \cite{SlT2}.   

There are other general tools developed here that are of independent interest. We prove that every cotorsion pair $(\mathcal A,\mathcal B)$ where $\mathcal B$ is closed under direct limits is a complete cotorsion pair of countable type with $\mathcal B$ actually being definable (Theorem \ref{t:counttype}). This is a non-hereditary version of the Countable Telescope Conjecture for module categories from \cite[Theorem 3.5]{SS}.

\medskip
This paper is organized as follows. After some preliminaries in Section 2, we start out in Section 3 by discussing the case of flat Mittag-Leffler modules.  Sections 4 and 5 are devoted to relative Mittag-Leffler conditions and their role in connections with vanishing of Ext and pure-injectivity. In Section 6 we consider  cotorsion pairs $(\mathcal A,\mathcal B)$ where $\mathcal B$ is closed under direct limits  and prove that they are of countable type. 

\section{Preliminaries}
\label{sec:prelim}

Let $R$ be an (associative unital) ring. We denote by $\ModR$ the category of all (right $R$-) modules.

\subsection{Bass modules.} Given a class $\mathcal C$ of countably presented modules, we call a module $M$ a \emph{Bass module} over $\mathcal C$, provided that $M$ is the direct limit of a countable direct system  
$$C_0 \overset{f_0}\to C_1 \overset{f_1}\to \dots \overset{f_{n-1}}\to C_n \overset{f_n}\to C_{n+1} \overset{f_{n+1}}\to \dots$$
where $C_n \in \mathcal C$ for each $n < \omega$. All Bass modules over $\mathcal C$ are countably presented; in fact, they are just the countable direct limits of the modules from $\mathcal C$, cf.\ \cite[\S5]{SlT}.

\subsection{Pure-injective modules.} For a (left, right) module $M$, we denote by $M^c = \Hom _{\Z}(M,\Q/\Z)$ the \emph{character} (right, left) module of $M$. A module $M$ is \emph{pure-injective} provided that the canonical embedding of $M$ into $M^{cc}$ splits. 
The pure-injective hull of a module $M$ will be denoted by $PE(M)$. 

 
\subsection{Roots of Ext.} For a class of modules $\mathcal C \subseteq \ModR$, we define its left and right Ext-orthogonal classes by ${}^{\perp} \mathcal C = \Ker\Ext ^1_R(-,\mathcal C)$ and $\mathcal C ^{\perp} = \Ker\Ext ^1_R(\mathcal C,-)$. 
For $\mathcal C = \ModR$, we have ${}^{\perp} \mathcal C = \mathcal P _0$ and $\mathcal C ^{\perp} = \mathcal I _0$, the classes of all projective and injective modules, respectively. 

\subsection{Deconstructible classes.} Given a module $M$ and an ordinal number $\sigma$, we call an ascending chain $\mathcal M = (M_\alpha \mid \alpha \leq \sigma)$ of submodules of $M$ a \emph{filtration} of $M$, if $M_0 = 0$, $M_\sigma = M$, and $\mathcal M$ is \emph{continuous}---that is, $\bigcup _{\alpha < \beta} M_\alpha = M_\beta$ for each limit ordinal $\beta\leq\sigma$. Moreover, given a class of modules $\mathcal C$, we call $\mathcal M$ a \emph{$\mathcal C$-filtration} of $M$, provided that each of the consecutive factors $M_{\alpha+1}/M_\alpha$ ($\alpha<\sigma$) is isomorphic to a module from $\mathcal C$. A module $M$ is \emph{$\mathcal C$-filtered}, if it admits a $\mathcal C$-filtration. 

Given a class $\mathcal C$ and a cardinal $\kappa$, we use $\mathcal C^{\leq\kappa}$ and $\mathcal C^{<\kappa}$ to denote the subclass of $\mathcal C$ consisting of all $\leq \kappa$-presented and $<\kappa$-presented modules, respectively. 

Let $\kappa$ be an infinite cardinal. A class of modules $\mathcal C$ is \emph{$\kappa$-deconstructible} provided that each module $M \in \mathcal C$ is $\mathcal C^{<\kappa}$-filtered. For example, the class $\mathcal P _0$ is $\aleph_1$-deconstructible by a classic theorem of Kaplansky. A class $\mathcal C$ is \emph{deconstructible} in case it is $\kappa$-deconstructible for some infinite cardinal $\kappa$. 

\subsection{Cotorsion pairs.}
A pair of classes of modules $\mathfrak C = (\mathcal A,\mathcal B)$ is a \emph{cotorsion pair} provided that $\mathcal A = {}^\perp \mathcal B$ and $\mathcal B = \mathcal A ^\perp$. The class $\Ker \mathfrak C = \mathcal A \cap \mathcal B$ is the \emph{kernel} of $\mathfrak C$. Notice that the class $\mathcal A$ is always closed under arbitrary direct sums and contains $\mathcal P _0$. Dually, the class $\mathcal B$ is closed under direct products and it contains $\mathcal I _0$. If moreover $\Ext ^2_R(A,B) = 0$ for all $A \in \mathcal A$ and $B \in \mathcal B$, then $\mathfrak C$ is a \emph{hereditary} cotorsion pair. The latter is equivalent to $\mathcal A$ being closed under kernels of epimorphisms. 

For example, for each $n < \omega$, there exist hereditary cotorsion pairs of the form $(\mathcal P _n, \mathcal P _n ^\perp)$ and $(^\perp \mathcal I _n, \mathcal I _n)$, where $\mathcal P _n$ and $\mathcal I _n$ denote the classes of all modules of projective and injective dimension $\leq n$, respectively.      

The cotorsion pair $\mathfrak C$ is said to be \emph{generated} by a class $\mathcal S$ provided that $\mathcal B = \mathcal S ^\perp$. In this case, if $\kappa$ is an infinite regular cardinal such that each module in $\mathcal S$ is $< \kappa$-presented, then $\mathcal A$ is $\kappa$-deconstructible. If $\mathfrak C$ is generated by a class $\mathcal S$ consisting of countably (finitely) generated modules with countably (finitely) presented first syzygies, then $\mathfrak C$ is said to be \emph{of countable (finite) type}\footnote{In the literature, the requirement is usually extended from the first one to all syzygies. However, this definition does not seem reasonable for non-hereditary cotorsion pairs, i.e. the ones we want to investigate. Our definition fits nicely into the crucial Theorem~\ref{t:counttype}, while it still ensures that finite type of $\mathfrak C$ implies definability of $\mathcal B$.}. If $R$ is right coherent, then it is equivalent to the statement `$\mathfrak C$ is generated by a class consisting of countably (finitely) presented modules'. Dually, $\mathfrak C$ is said to be \emph{cogenerated} by a class $\mathcal C$ provided that $\mathcal A = {}^\perp \mathcal C$.

\subsection{Approximations.}
A class $\mathcal C$ of modules is \emph{precovering}, if for each module $M$ there exists a morphism $f \in \Hom _R(C,M)$ with $C \in \mathcal C$, such that each morphism $f^\prime \in \Hom _R(C^\prime,M)$ with $C^\prime \in \mathcal C$ factors through $f$. Such $f$ is called a \emph{$\mathcal C$-precover} of the module $M$. 

A precovering class of modules $\mathcal C$ is called \emph{special precovering} provided that each module $M$ has a $\mathcal C$-precover $f : C \to M$ which is surjective and satisfies $\mbox{Ker} (f) \in \mathcal C^\perp$. Moreover, $\mathcal C$ is called \emph{covering} provided that each module $M$ has a $\mathcal C$-precover $f : C \to M$ with the following minimality property:  $g$ is an automorphism of $C$, whenever $g : C \to C$ is an endomorphism of $C$ with $fg = f$. Such $f$ is called a \emph{$\mathcal C$-cover} of $M$. Dually, we define \emph{preenveloping}, \emph{special preenveloping}, and \emph{enveloping} classes of modules. 

We will also deal with preenveloping and precovering classes in the category of all finitely presented modules -- these classes are usually called \emph{covariantly finite} and \emph{contravariantly finite}, respectively.  

A cotorsion pair $\mathfrak C = (\mathcal A, \mathcal B)$ is called \emph{complete}, provided that $\mathcal A$ is a special precovering class (or, equivalently, $\mathcal B$ a special preenveloping class). 
For example, each cotorsion pair generated by a set of modules is complete. Also, every deconstructible class of modules is precovering provided it is closed under transfinite extensions. Further, $\mathfrak C$ is \emph{closed}, provided that $\mathcal A = \varinjlim \mathcal A$. If $\mathfrak C$ is closed and complete, then $\mathcal A$ is a covering class in $\ModR$, cf.\ \cite[Part II]{GT}.

\subsection{Locally free modules.}
The following less well known notation from \cite{HT} and \cite{SlT} will be convenient:

\begin{defn}\label{d:dense} Let $R$ be a ring and $\lambda$ an infinite regular cardinal. 

A system $\mathcal S$ consisting of $<\!\!\lambda$-presented submodules of a module $M$ satisfying the conditions 
\begin{enumerate}
\item $\mathcal S$ is closed under unions of well-ordered ascending chains of length $<\lambda$, and
\item each subset $X \subseteq M$ \st $\vert X\vert < \lambda$ is contained in some $N \in \mathcal S$,
\end{enumerate}
is called  a \emph{$\lambda$-dense system} of submodules of $M$.
\end{defn}

Of course, $M$ is then the directed union of these submodules.

\smallskip
Let $\mathcal F$ be a set of countably presented modules. Denote by $\mathcal C$ the class of all modules possessing a countable $\mathcal F$-filtration. 
A module $M$ is \emph{locally $\mathcal F$-free} provided that $M$ contains an $\aleph _1$-dense system of submodules from $\mathcal C$. 
Notice that if $M$ is countably presented, then $M$ is locally $\mathcal F$-free, if and only if $M \in \mathcal C$. 

\smallskip
We will mainly be interested in the case when $\mathcal F=\mathcal A^{\le \omega}$ for a cotorsion pair $\mathfrak C=(\mathcal A,\mathcal B)$. Then $\mathcal C= \mathcal A^{\le \omega}$, and a module is locally $\mathcal A^{\le \omega}$-free if and only if it admits an $\aleph _1$-dense system of countably presented submodules from $\mathcal A$.  For a different description of  locally $\mathcal A^{\le \omega}$-free modules, see \cite[Theorem 4.4]{AST2}.

We note that by \cite{HT}, the locally $\mathcal P _0^{\le \omega}$-free modules coincide with the flat Mittag-Leffler modules from Definition \ref{d:ML} below. 

\subsection{Filter-closed classes.} 
For every directed set $(I,\leq)$, there is an \emph{associated filter} $\mathfrak F_I$ on $(\Pot(I),\subseteq)$, namely the one with a basis consisting of the upper sets $\uparrow\! i = \{j \in I \mid j \geq i \}$ for all $i \in I$, that is
$$ \mathfrak F_I = \{X \subseteq I \mid (\exists i \in I)(\uparrow\! i \subseteq X) \}. $$

\begin{defn} \label{def:f_prod}
Let $\mathfrak F$ be a filter on the power set $\Pot(X)$ for some set $X$, and let $\{M_x\mid x\in X\}$ be a set of modules. Set $M = \prod_{x\in X}M_x$. Then the \emph{$\mathfrak F$-product} $\Sigma_\mathfrak{F} M$ is the submodule of $M$ \st
$$ \Sigma_\mathfrak{F} M = \{m \in M \mid z(m) \in \mathfrak{F} \} $$
where for an element $m = (m_x\mid x \in X) \in M$, we denote by $z(m)$ its zero set $\{x \in X \mid m_x = 0\}$. 

For example, if $\mathfrak F = \Pot(X)$ then $\Sigma_\mathfrak{F} M = M$ is just the direct product, while if $\mathfrak F$ is the Fr\' echet filter on $X$ then $\Sigma_\mathfrak{F} M = \bigoplus_{x \in X} M_x$ is the direct sum.   
The module $M / \Sigma_\mathfrak{F} M$ is called an \emph{$\mathfrak F$-reduced product}. Note that for $a,b \in M$, we have the equality $\bar a = \bar b$ in the $\mathfrak F$-reduced product \iff $a$ and $b$ agree on a set of indices from the filter $\mathfrak F$. 

In the case that $M_x = M_y$ for every pair of elements $x,y\in X$, we speak of an \emph{$\mathfrak F$-power} and an \emph{$\mathfrak F$-reduced power} (of the module $M_x$) instead of an $\mathfrak F$-product and an $\mathfrak F$-reduced product, respectively. If moreover $\mathcal F$ is an ultrafilter on $X$, then 
the $\mathfrak F$-reduced power is called an \emph{ultrapower} of $M_x$. 

Finally, a nonempty class of modules $\mathcal G$ is called \emph{filter-closed}, if it is closed under arbitrary $\mathfrak F$-products (for any set $X$ and any arbitrary filter $\mathfrak F$ on $\Pot (X)$). 
Notice that a class of modules is filter-closed in case it is closed under direct products and direct limits (of direct systems consisting of monomorphisms), see \cite[Lemma 3.3.1]{P2}.
\end{defn}

\section{Flat Mittag-Leffler modules and approximations}
\label{sec:prototype}

We start with the notion of a flat Mittag-Leffler module studied since the classic works of Grothendieck \cite{Gr} and Raynaud-Gruson \cite{RG}. From its many facets, we choose the one involving tensor products for a definition: 

\begin{defn}\label{d:ML} Let $R$ be a ring. A module $M$ is \emph{flat Mittag-Leffler} provided that the functor $M \otimes _R - : \RMod \to \hbox{Mod-}{\mathbb Z}$ is exact, and the canonical group homomorphism $\varphi : M \otimes _R \prod_{i \in I} Q_i \to \prod_{i \in I} M \otimes_R Q_i$ defined by $\varphi (m \otimes_R (q_i)_{i \in I}) = (m \otimes_R q_i)_{i \in I}$ is monic for each family of left $R$-modules $( Q_i \mid i \in I )$. 

The class of all flat Mittag-Leffler modules will be denoted by $\mathcal F \mathcal M$.  
\end{defn}   

If $R$ is a right perfect ring, then flat Mittag-Leffler modules coincide with the projective modules, and form a covering class of modules by a classic theorem by Bass.   
However, if $R$ is not right perfect (e.g., if $R$ is right noetherian, but not right artinian), then the class $\mathcal F \mathcal M$ is not deconstructible \cite[Corollary 7.3]{HT}.

In \cite[Theorem 1.2(iv)]{ST}, the Singular Cardinal Hypothesis (SCH) was used to prove that $\mathcal F \mathcal M$ is not even  precovering when $R$ has cardinality $\leq \aleph_0$. The assumption of SCH was removed in \cite{BS}. 

Our goal in this section is to remove the cardinality restriction on $R$ as well, and thus to obtain  a basic example of a non-precovering class of modules over an arbitrary non-right perfect ring $R$.  

\medskip
The following lemma is formulated for the more general setting of locally $\mathcal F$-free modules (see Definition \ref{d:dense}): 
 
\begin{lem} \label{l:nonprec} Let $\mathcal F$ be a class of countably presented modules, and $\mathcal L$ the class of all locally $\mathcal F$-free modules. Let $N$ be a Bass module over $\mathcal F$. Assume that $N$ is not a direct summand in a module from $\mathcal L$. Then $N$ has no $\mathcal L$-precover.
\end{lem}
\begin{proof} As the direct limit, $N$ is an epimorphic image of a direct sum of modules from~$\mathcal F$. Since any direct sum of modules from $\mathcal F$ belongs to $\mathcal L$ (cf. \cite[Lemma~4.2]{SlT}), all the potential $\mathcal L$-precovers of $N$ have to be surjective.

For the sake of contradiction, let us assume that
$$0 \longrightarrow M \overset{m}\longrightarrow A \overset{f}\longrightarrow N \longrightarrow 0$$
is a short exact sequence where $f$ is an $\mathcal L$-precover of $N$. Let $\kappa$ be an infinite cardinal \st $|R|\leq\kappa$ and $|M|\leq 2^\kappa = \kappa ^\omega$. By \cite[Lemma 5.6]{SlT}, there exists a short exact sequence 
\begin{equation}\label{e:1} 
0 \longrightarrow D \overset{\subseteq}\longrightarrow L \longrightarrow {N^{(2^\kappa)}} \longrightarrow 0 
\end{equation}
such that $L \in \mathcal L$ and $D$ is a direct sum of $\kappa$ modules from $\mathcal F$.

Applying $\Hom _R(-,m)$ to (\ref{e:1}), we obtain the following commutative diagram with exact rows
$$\begin{CD}
	\Hom_R(D,M)	@>{\delta}>>	\Ext_R^1(N,M)^{2^\kappa}  @>>> \Ext_R^1(L,M)  \\ 
	 @V{\Hom_R(D,m)}VV	@V{\Ext_R^1(N,m)^{2^\kappa}}VV  @V{\Ext_R^1(L,m)}VV \\
  \Hom_R(D,A)  @>>>	\Ext_R^1(N,A)^{2^\kappa}  @>>> \Ext_R^1(L,A).
\end{CD}$$

Since $L\in\mathcal L$ and $f$ is an $\mathcal L$-precover, the map $\Ext_R^1(L,m)$ is monic. It follows that $\Ker(\Ext_R^1(N,m))^{2^\kappa}\subseteq \Img(\delta)$. However, by our assumption, $2^\kappa\geq|\Hom_R(D,M)|\geq|\Img(\delta)|$, and so $\Ext_R^1(N,m)$ has to be monic as well. The latter is equivalent to $\Hom_R(N,f)$ being onto. In particular, $f$ splits---a contradiction.
\end{proof}

Now, we can easily prove our first main result:

\begin{thm} \label{t:nonprec} Let $R$ be an arbitrary ring. Then $\mathcal F \mathcal M$ is precovering, if and only if $R$ is right perfect. \end{thm}
\begin{proof} The if part is well known: $\mathcal F \mathcal M =\mathcal{P}_0$ whenever $R$ is a right perfect ring. 

Assume $R$ is not right perfect. By a classic result of Bass, there exists a countably presented flat module (= a Bass module over $\mathcal P _0 ^{< \omega}$) $N$ such that $N \notin \mathcal P _0$. By \cite[Corollary 2.10(i)]{HT}, the class $\mathcal F \mathcal M$ coincides with the class of all locally $\mathcal F$-free modules for $\mathcal F = \mathcal P_0^{\leq \omega}$. In particular, $N \notin \mathcal F \mathcal M$.
So Lemma~\ref{l:nonprec} applies, and shows that $N$ has no $\mathcal F \mathcal M$-precover.
\end{proof}
 
For further applications of Lemma \ref{l:nonprec}, we refer to \cite[\S4]{AST2}, \cite[\S3]{SlT2} and \cite[\S4]{S}.

\medskip
\section{Stationarity and the vanishing of Ext}
\label{sec:vanish} 
 
Next, we develop several technical tools involving Mittag-Leffler conditions and pure injectivity. It turns out that these tools are available in a rather general setting,  but one has to work with relative Mittag-Leffler conditions rather than the {\lq}absolute{\rq} ones. More precisely, we will use $\mathcal B$-stationary modules, where $\mathcal B$ will be a class closed under direct products and direct limits.

In the present Section, we review this concept with a special emphasis on its relationship to the vanishing of Ext. In Section 5 we will see that the classes $\mathcal B$ as above admit a pure-injective cogenerator $C$, and we will thus  focus on  $C$-stationarity.

Let us first recall the classic notions related to inverse systems of modules (see \cite{AH}, \cite{RG}):

\begin{defn}\label{d:MLc} Let $R$ be a ring and $\mathcal H = (H_i, h_{ij} \mid i \leq j \in I )$ be an inverse system of modules, with the inverse limit $(H, h_i \mid i \in I )$. 

Then $\mathcal H$ is a \emph{Mittag-Leffler} inverse system, provided that for each $k \in I$ there exists $k \leq j \in I$ such that $\mbox{Im}(h_{kj}) = \mbox{Im}(h_{ki})$ for each $j \leq i \in I$, that is, the terms of the decreasing chain $( \mbox{Im}(h_{ki}) \mid k \leq i \in I )$ of submodules of $H_k$ stabilize. If moreover the stabilized term $\mbox{Im}(h_{kj})$ equals $\mbox{Im}(h_k)$, then $\mathcal H$ is called \emph{strict Mittag-Leffler}. Notice that the two notions coincide when $I$ is countable.  

Let $B$ be a module and $\mathcal M = (M_i, f_{ji} \mid i \leq j \in I )$ be a direct system of finitely presented modules, with the direct limit $(M, f_i \mid i \in I )$. Applying the contravariant functor $\Hom _R(-,B)$, we obtain the induced inverse system $\mathcal H = (H_i, h_{ij} \mid i \leq j \in I )$, where $H_i = \Hom _R(M_i,B)$ and $h_{ij} = \Hom _R(f_{ji},B)$ for all $i \leq j \in I$. 

Let $\mathcal B$ be a class of modules. A module $M$ is \emph{$\mathcal B$-stationary} (\emph{strict $\mathcal B$-stationary}), provided that $M$ can be expressed as the direct limit of a direct system $\mathcal M$ of finitely presented modules so that for each $B \in \mathcal B$, the induced inverse system $\mathcal H$ is Mittag-Leffler (strict Mittag-Leffler).  

If $\mathcal B = \{ B \}$ for a module $B$, we will use the notation \emph{$B$-stationary} instead of $\{ B \}$-stationary, and similarly for the strict stationarity.
\end{defn} 

\begin{rem}\label{r:coinc}  Let $B$ be a module. Then the strict $B$-stationarity of $M$ can equivalently be expressed as follows: for each $l \in I$, there exists $l \leq i \in I$ such that $\Img(\Hom _R(f_{il},B)) \subseteq \Img(\Hom _R(f_l,B))$, see \cite[\S8]{AH}. Moreover, the notions of a $B$-stationary and strict $B$-stationary module coincide in the case when $M$ is countably presented. They also coincide for $B$ (locally) pure-injective by \cite[Proposition 1.7]{H}. 

In fact, if $M$ is $B$-stationary (strict $B$-stationary), then the induced inverse system $\mathcal H$ is Mittag-Leffler (strict Mittag-Leffler) for \emph{each} presentation of $M$ as the direct limit of a direct system $\mathcal M$ of finitely presented modules, see \cite{AH}.    

Also, $\mathcal F \mathcal M$ coincides with the class of all flat $R$-stationary modules, cf.\ \cite{RG}. 
\end{rem}      

First, we will deal with the strict $\mathcal B$-stationarity of the modules in $^\perp \mathcal B$. The following result is a mix of Lemmas 2.3 and 2.5 from \cite{SS}; for the reader's convenience, we provide a detailed proof here:

\begin{lem} \label{l:strictstat} Let $\mathcal B$ be a filter-closed class of modules. 
Then every module $M\in{}^\perp\mathcal B$ is strict $\mathcal B$-stationary.
  \end{lem}
\begin{proof} We will prove the following: 
if $(M, f_i\mid i\in I)$ is the direct limit of a direct system $\mathcal M = (M_i, f_{ji}\mid i\leq j \in I)$ consisting of finitely generated modules, then for each $l \in I$ there exists $l \leq i \in I$ such that for each $B\in\mathcal B$ and $g\in\Hom_R(M_i,B)$, we have $gf_{il}\in\Img(\Hom_R(f_l,B))$.

Suppose that the claim is not true, so there exists $l \in I$ such that for each $l \leq i \in I$ there exist $B_i\in\mathcal B$ and $g_i:M_i\to B_i$, such that $g_if_{il}$ does not factor through $f_l$. For $i \in I$ such that $l \nleq i$, we let $B_i = 0$ and $g_i = 0$. Put $B = \prod_{i \in I} B_i$.

We define a homomorphism $h_{ji}: M_i \to B_j$ for each pair $i,j \in I$ in the following way: $h_{ji} = g_jf_{ji}$ if $i \leq j$ and $h_{ji} = 0$ otherwise. This family of maps gives rise to the canonical homomorphism $h: \bigoplus_{k \in I} M_k \to B$. More precisely, if we denote by $\pi_j: B \to B_j$ the canonical projection and by $\nu_i: M_i \to \bigoplus _{k\in I}M_k$ the canonical inclusion, $h$ is the (unique) map such that $\pi_j h \nu_i = h_{ji}$. Note that for every $i,j \in I$ \st $i \leq j$, the set $\{k \in I \mid h_{ki} = h_{kj}f_{ji} \}$ is in the associated filter $\mathfrak F_I$ since it contains each $k \geq j$. 
Hence, if we denote by $\varphi$ the canonical pure epimorphism $\bigoplus _{i\in I}M_i\to M = \varinjlim _{i\in I}M_i$ (\st $\varphi\nu_i = f_i$ for all $i\in I$), then $h(\Ker (\varphi)) \subseteq \Sigma _{\mathfrak F_I}B$. So there is a well-defined homomorphism $u$ from $M$ to the $\mathfrak F_I$-reduced product $B/\Sigma _{\mathfrak F_I}B$ making the following diagram commutative ($\rho$ denotes the canonical projection):
$$\begin{CD}
	B @>{\rho}>>	B/\Sigma _{\mathfrak F_I}B @>>> 0\ 	\\
	@A{h}AA			@A{u}AA	\\
	\bigoplus _{i\in I} M_i	@>{\varphi}>>	M @>>> 0.
\end{CD}$$
Since $\mathcal B$ is filter-closed, $\Sigma _{\mathfrak F_I}B\in\mathcal B$, whence $\Ext^1_R(M,\Sigma_{\mathfrak F_I} B) = 0$. It follows that there exists $g \in \Hom_R(M,B)$ \st $u = \rho g$. 

For each $i \in I$, we have $\rho g f_i = \rho g \varphi \nu_i = \rho h \nu_i$; so $g f_i - h \nu_i$ maps $M_i$ into $\Sigma _{\mathfrak F_I}B$. Since $M_i$ is finitely generated, there exists $i \leq j \in I$ such that $\pi_k g f_i = \pi_k h \nu_i$ for all $j \leq k \in I$. However, $\pi_k h \nu_i = h_{ki} = g_k f_{ki}$. In particular, for $i = l$ and $k = j$, we infer that $\pi_j g f_l = g_j f_{jl}$. Thus, $g_jf_{jl}$ factors through $f_l$, a contradiction.           
\end{proof}

We will also need the following variant of \cite[Proposition 2.7]{SS}.

\begin{prop} \label{p:pureinprod} Let $\mathcal G$ be a class of modules and $M$ a countably presented module such that $M\in {}^\perp \mathcal G$. Then the following is equivalent:
\begin{enumerate}
\item $M\in {}^\perp D$ for each module $D$ isomorphic to a pure submodule of a product of modules from $\mathcal G$;
\item $M\in {}^\perp D$ for each module $D$ isomorphic to a countable direct sum of modules from $\mathcal G$;
\item $M$ is $\mathcal G$-stationary.
\end{enumerate}
\end{prop}
\begin{proof} Since direct sums are pure in the corresponding direct products, the implication $(1) \Rightarrow (2)$ is trivial.

The implication $(2) \Rightarrow (1)$ is exactly \cite[Proposition 2.7]{SS} (for $\mathcal G$ replaced by its closure under countable direct sums). Its proof in \cite{SS} proceeds by showing that $(2) \Rightarrow (4) \Rightarrow (1)$, where $(4)$ says that $M$ is $D$-stationary for any module $D$ isomorphic to a pure submodule of a product of modules from $\mathcal G$. However, $(4)$ is equivalent to $(3)$ by \cite[Corollary 3.9]{AH}, whence $(2) \Rightarrow (3) \Rightarrow (1)$. 
\end{proof}

The next lemma on strict $B$-stationarity is inspired by \cite[Lemma 3.15]{H}. 

\begin{lem} \label{l:kerepi} Let $B$ be a module and $0\to N 
\to A\to M\to 0$ a short exact sequence of modules such that  $M\in{}^\perp (B^{(I)})^{cc}$ for each set $I$. Then $N$ is strict $B$-stationary if so is $A$. 
\end{lem}
\begin{proof} According to \cite[Theorem 8.11]{AH} (see also \cite{Z}), a necessary and sufficient condition for the strict $B$-stationarity of $N$ is that for each set $I$, the canonical map $$\nu :N\otimes_R\Hom_{\Z}(B,(\Q/\Z)^I) \to \Hom_{\Z}(\Hom_R(N,B),(\Q/\Z)^I)$$ defined by $\nu (n \otimes f)(g) = f(g(n))$ is injective. 

Since $\Hom_{\Z}(B,(\Q/\Z)^I)\cong (B^{(I)})^c$, we have $\Tor_1^R(M,\Hom_{\Z}(B,(\Q/\Z)^I)) = 0$ by our assumption on $M$. From the commutative diagram
$$\begin{CD}
	0	@>>>	N\otimes_R\Hom_{\Z}(B,(\Q/\Z)^I) @>>> A\otimes_R\Hom_{\Z}(B,(\Q/\Z)^I)  \\ 
	@.		@V{\nu}VV	 @V{\alpha}VV \\
   @.	\Hom_{\Z}(\Hom_R(N,B),(\Q/\Z)^I) @>>> \Hom_{\Z}(\Hom_R(A,B),(\Q/\Z)^I),
\end{CD}$$
\noindent we infer that $\nu$ is injective, because $\alpha$ is injective by our assumption on $A$.
\end{proof}

\bigskip
\section{Stationarity and pure-injectivity}
\label{sec:pinj} 

The classes of pure-injective modules especially relevant here are the $\Sigma$-pure injectives, and the elementary cogenerators. We will not deal with the model theoretic background of these notions; we just note that, by a classic theorem of Frayne, if two modules $M$ and $N$ are elementarily equivalent, then $M$ is a pure submodule of an ultrapower of $N$. The pure-injective hull $PE(M)$ of a module $M$ is elementarily equivalent to $M$ by a theorem of Eklof and Sabbagh, and so is its double dual $M^{cc}$, cf.\ \cite{P}.

\smallskip
The relation between stationarity and $\sum$-pure-injectivity goes back to work by Zimmermann \cite{Z}.

\begin{lem} \label{l:zimmer} For a module $C$, the following conditions are equivalent:
\begin{enumerate}
\item $C$ is $\Sigma$-pure-injective;
\item all modules are strict $C$-stationary;
\item all modules are $C$-stationary.
\item all countably presented modules are $C$-stationary.
\end{enumerate}
\end{lem}
\begin{proof} The equivalence of the first two conditions comes from \cite[Theorem 3.8]{Z}, and $(2)\Rightarrow (3)\Rightarrow (4)$ are trivial. For the implication $(4)\Rightarrow (3)$, see \cite[Proposition 3.10]{AH}. It remains to prove that $(3)$ implies $(1)$.

Applying \cite[Corollary 3.9]{AH}, we get that all modules are $B$-stationary, where $B$ is any pure submodule of a pure-epimorphic image of a direct product of copies of $C$. In particular, by \cite[Lemma 3.3.1]{P2}, this holds for any pure-injective module $B$ which is elementarily equivalent to $C$ (e.g., for $B = PE(C)$). Since $B$ is pure-injective, it follows from Remark \ref{r:coinc} that all modules are even strict $B$-stationary, and $B$ is $\Sigma$-pure-injective by the above. Thus $C$ is $\Sigma$-pure-injective, because $C$ is a pure submodule of $B$.
\end{proof}

We now turn to elementary cogenerators.
        
\begin{defn}\label{d:elem} A pure-injective module $E$ is called an \emph{elementary cogenerator}, if every pure-injective direct summand of a module elementarily equivalent to $E^{\aleph_0}$ is a direct summand of a direct product of copies of $E$. 
\end{defn}

Notice that by \cite[Corollary 9.36]{P}, for each module $M$ there exists an elementary cogenerator which is elementarily equivalent to $M$.
This allows to prove the following result which will play an important role in the sequel.

\begin{lem}\label{l:belem} Let $\mathcal B$ be a class of modules closed under direct products and direct limits. Then $\mathcal B$ contains a pure-injective module $C$, such that each module $B \in\mathcal B$ is a pure submodule of a direct product of copies of $C$. Moreover, if $\mathcal B$ contains a cogenerator, then $C$ is a cogenerator for $\ModR$.
\end{lem}       
\begin{proof} First, it is a well-known fact that $\mathcal B = \varinjlim\mathcal B$ yields that $\mathcal B$ is closed under direct summands.

Now the closure properties of $\mathcal B$ imply that if $B\in\mathcal B$ is elementarily equivalent to a pure-injective module $A$, then also $A\in\mathcal B$. In particular, $\mathcal B$ is closed under taking pure-injective hulls, and double duals.

Thus we can choose among the modules from $\mathcal B$ a representative set for elementary equivalence, $\mathcal S$, consisting of elementary cogenerators. Let $C = \prod\mathcal S$. Then $C\in\mathcal B$ is pure-injective, and it has the property that every module $B\in\mathcal B$ is isomorphic to a pure submodule in a direct product of copies of $C$. Indeed, we have $B^{cc}\in\mathcal B$; moreover $B^{cc}$ is a pure-injective direct summand of $(B^{cc})^{\aleph_0}$ which is a module elementarily equivalent to $E^{\aleph_0}$ for some $E\in\mathcal S$. Hence $B^{cc}$ is a direct summand in a direct product, $D$, of copies of $C$ (by Definition \ref{d:elem}), and $B$ is pure in $B^{cc}$, and hence in $D$. 

Moreover, if $B \in \mathcal B $ is a cogenerator for $\ModR$, then $C$ cogenerates $\ModR$ as well.
\end{proof}

Given a pure-injective cogenerator $C$ as above, one can use the following lemma to find in  any (strict) $C$-stationary module a rich supply of countably presented $C$-stationary submodules. 

\begin{lem} \label{l:Cext} Let $C$ be a pure-injective module which cogenerates $\ModR$, and $M$ be a strict $C$-stationary module. Then there exists an $\aleph_1$-dense system $\mathcal L$ of strict $C$-stationary submodules of $M$ \st  $$\Hom_R(M,C)\to\Hom_R(N,C)\hbox{ is surjective}\eqno{(\dagger)}$$ 
for every directed union $N$ of modules from $\mathcal L$. 
\end{lem}
\begin{proof} Let $Q = C^c$. By \cite[Proposition 8.14(2)]{AH}, a module $A$ is strict $C$-stationary if and only if it is $Q$-Mittag-Leffler. Repeatedly using \cite[Theorem 5.1(4)]{AH} for $M$, we obtain a $\subseteq$-directed set $\mathcal F$ of countably presented $Q$-Mittag-Leffler submodules of $M$ satisfying $(2)$ from Definition~\ref{d:dense} for $\lambda = \aleph _1$: observe that the map $v$ in the statement of \cite[Theorem 5.1(4)]{AH} is monic since the injectivity of $v\otimes_R Q$ implies $(\dagger)$ (we use that $C$ is a direct summand in $Q^c$) and $C$ is a cogenerator.

We extend $\mathcal F$ gradually by adding countable directed unions, noticing along the way that each newly added countably presented module $N$ has the property that $N\subseteq M$ stays monic after applying $-\otimes_R Q$ (as $\Tor^R_1(-,Q)$ commutes with direct limits), and it is $Q$-Mittag-Leffler by \cite[Corollary 5.2]{AH}. In this way, we eventually arrive at the desired $\aleph_1$-dense system $\mathcal L$ of strict $C$-stationary submodules of $M$. Note that any directed union of modules from $\mathcal L$ satisfies $(\dagger)$ since $(\dagger)$ is implied by the injectivity of the corresponding tensor map, which, in turn, is a property preserved by taking directed unions.
\end{proof}

\begin{rem} \label{r:lambda} Lemma~\ref{l:Cext} holds with $\aleph_1$ replaced by any regular uncountable cardinal. However, we won't need it in this generality.
\end{rem}

A useful tool for proving that many classes of the form $^\perp \mathcal B$ are deconstructible is provided by 

\begin{lem} \label{l:up} \cite[Proposition 2.4]{SS} \, Let $\mathcal B$ be a filter-closed class of modules \st $\mathcal B$ cogenerates $\ModR$. Then for each uncountable regular cardinal $\lambda$ and each module $M\in{}^\perp \mathcal B$, there is a $\lambda$-dense system $\mathcal C_\lambda$ of submodules of  $M$ \st $M/N\in{}^\perp \mathcal B$ for all $N \in \mathcal C_\lambda$.
\end{lem}

\section{Countable type}
\label{sec:ctype}

Let $\mathfrak C=(\mathcal A, \mathcal B)$ be a hereditary cotorsion pair   in $\ModR$. The Telescope Conjecture  for Module Categories asserts that  $\mathfrak C$  is of finite type whenever $\mathcal A$  and $\mathcal B$ are closed under direct limits. This statement is  known to be true  in some special cases  \cite{AST}. On the other hand,  the Countable Telescope Conjecture was proved  in full generality in \cite[Theorem 3.5]{SS}.  It states that  $\mathfrak C$   is of countable type whenever $\mathcal B$ is closed under  unions of well-ordered chains. 

As an application of the tools developed above, we prove a version of the Countable Telescope Conjecture for not necessarily hereditary cotorsion pairs. 

\begin{thm} \label{t:counttype} Let $\mathfrak C = (\mathcal A,\mathcal B)$ be a cotorsion pair with $\varinjlim \mathcal B = \mathcal B$. Then $\mathfrak C$ is of countable type, and $\mathcal B$ is definable.
\end{thm}
\begin{proof} Let $C$ be the pure-injective module constructed for $\mathcal B$ in Lemma \ref{l:belem}.

Let $\mathcal A_0 = \mathcal A ^{\leq \omega}$. By induction on $\kappa$, we are going to prove that each $\kappa$-presented module $M\in\mathcal A$ is $\mathcal A_0$-filtered. There is nothing to prove for $\kappa \leq \aleph_0$, so let $\kappa$ be uncountable.

In the $\kappa$-presented module $M\in\mathcal A$, we will construct by induction, for each uncountable regular cardinal $\lambda\leq\kappa$, a $\lambda$-dense system  $\mathcal C_\lambda$ of submodules of $M$ such that $M/N \in \mathcal A$ for each $N \in \mathcal C_\lambda$, and   
\begin{equation}\label{e:2}
\mathcal C_\lambda \subseteq \mathcal A.
\end{equation}

Our strategy is to start with a $\mathcal C_\lambda$ given by Lemma~\ref{l:up} for $\mathcal G = \mathcal B$, and then select a suitable subfamily in $\mathcal A$. Note that it is enough to ensure that for every $N\in\mathcal C_\lambda$ there exists $L\in\mathcal C_\lambda$, such that $N\subseteq L \in \mathcal A$. Then the family $\mathcal C_\lambda\cap\mathcal A$ is the desired one, since each ascending chain in $\mathcal C_\lambda\cap\mathcal A$ is actually an $\mathcal A$-filtration and $\mathcal A$ is closed under filtrations. Indeed, for each $B \in \mathcal B$, if $N \in \mathcal C_\lambda$ and $L \in \mathcal A$ are such that $N \subseteq L$, then all homomorphisms from $N$ to $B$ extend to $M$, and hence to $L$; thus $L/N \in {}^\perp \{ B \}$.    

\medskip
We will distinguish the following four cases:

\smallskip

\noindent Case 1. \emph{$\lambda = \aleph _1$}: \, Fix a free presentation of $M$ 
$$\begin{CD} 0 @>>> K @>>> R^{(\kappa)} @>f>> M @>>> 0. \end{CD}$$
Let $\mathcal C_{\aleph_1}$ be an $\aleph _1$-dense system in $M$ provided by Lemma~\ref{l:up}. After taking the intersection with the set of all images $f(R^{(X)})$ where  $X$ is a countable subset of $\kappa$, we can assume that $\mathcal C_{\aleph_1}$ is compatible with this presentation, that is, each $N\in\mathcal C_{\aleph_1}$ has the form $f(R^{(X_N)})$ for a countable subset $X_N$ of $\kappa$. 

Let $\mathcal K = \{\,\Ker(f\restriction R^{(X_N)})\mid N\in\mathcal C_{\aleph_1}\}$. Since $\mathcal B$ is closed under coproducts and double duals, it follows from   Lemma~\ref{l:kerepi} that $K$ is strict $C$-stationary. Thus we can use Lemma~\ref{l:Cext} to obtain another $\aleph_1$-dense system, $\mathcal L$, this time consisting of submodules of $K$. Clearly, the system $\mathcal K\cap\mathcal L$ is $\aleph _1$-dense as well. Our new $\mathcal C_{\aleph_1}$ is defined as $\{N\in\mathcal C_{\aleph_1}\mid\Ker(f\restriction R^{(X_N)})\in\mathcal L\}$.

Notice that $\mathcal C_{\aleph_1}\subseteq {}^\perp C$. Indeed, given a module $N\in\mathcal C_{\aleph_1}$, each $h:\Ker(f\restriction R^{(X_N)})\to C$ can be extended to some $h^\prime: R^{(X_N)}\to C$ by the property $(\dagger)$ from Lemma~\ref{l:Cext} and by the fact that $M\in{}^\perp C$.

Let $N\in\mathcal C_{\aleph_1}$. Since $M/N \in \mathcal A$, $N$ is the kernel of the epimorphism $M\to M/N$ between two modules from~$\mathcal A$. By Lemma~\ref{l:strictstat}, $M$ is strict $C$-stationary, so using Lemma~\ref{l:kerepi} again, we see that each $N \in \mathcal C_{\aleph_1}$ is a countably presented $C$-stationary module from ${}^\perp C$.
Using Proposition~\ref{p:pureinprod} (for $\mathcal G = \{C\}$) together with the properties of $C$ guaranteed by Lemma~\ref{l:belem}, we conclude that $\mathcal C_{\aleph_1}\subseteq\mathcal A$.
\smallskip

\noindent Case 2. \emph{$\lambda$ weakly inaccessible}: \, As in the previous step, we start with a family $\mathcal C_\lambda$ provided by Lemma~\ref{l:up}. Since each $N_0\in\mathcal C_\lambda$ is $<\mu$-presented for some regular uncountable cardinal $\mu<\lambda$, we simply choose $N^\prime_0\in\mathcal C_\mu \subseteq \mathcal A$ containing $N_0$ as a submodule. Then there is $N_1\in\mathcal C_\lambda$ containing $N^\prime _0$, etc. The union, $N$, of the chain $N_0\subseteq N^\prime_0\subseteq N_1\subseteq N^\prime_1\subseteq\cdots$ satisfies $N \in \mathcal C_\lambda$. Moreover, $N \in \mathcal A$ since the chain $N^\prime_0\subseteq N^\prime_1\subseteq\cdots$ is an $\mathcal A$-filtration of $N$.
\smallskip

\noindent Case 3. \emph{$\lambda$ a successor of a regular cardinal}: \, $\lambda = \nu^+$ for a regular cardinal $\nu$. For each $N_0$ from $\mathcal C_\lambda$ (not necessarily satisfying condition $(\ref{e:2})$ above), we easily build a continuous chain $\mathcal N_0 = (N^0_\alpha\mid\alpha<\nu)$ of modules from $\mathcal C_\nu \subseteq \mathcal A$ so that $N_0 \subseteq \bigcup\mathcal N_0$. Again, the union is in $\mathcal A$. We continue by choosing $N_1\in\mathcal C_\lambda$ containing this union, and a chain $\mathcal N_1 = (N^1_\alpha\mid\alpha<\nu)$ in $\mathcal C_\nu$ \st $N^0_\alpha\subseteq N^1_\alpha$ for all $\alpha<\nu$, and $N_1\subseteq\bigcup\mathcal N_1$. We proceed further by taking $N_2\in\mathcal C_\lambda$, etc. Clearly, $N = \bigcup _{i<\omega}N_i \in \mathcal C_\lambda$. 
Furthermore, the ascending chain 
$$\bigcup _{i<\omega} N^i_0\subseteq\bigcup _{i<\omega} N^i_1\subseteq\cdots\subseteq\bigcup _{i<\omega} N^i_\alpha\subseteq\cdots$$
of submodules from $\mathcal C_\nu$ forms an $\mathcal A$-filtration of $N$, showing that  $N \in \mathcal A$.
 
\smallskip

\noindent Case 4. \emph{$\lambda$ a successor of a singular cardinal}: \, $\lambda = \nu^+$ where $\mu = \hbox{cf}(\nu)<\nu$. We choose a strictly increasing continuous chain $(\nu _\alpha\mid\alpha<\mu)$ of infinite cardinals which is cofinal in $\nu$, such that $\nu_0>\mu$. Let $N_0$ be arbitrary module from the family $\mathcal C_\lambda$ given by Lemma~\ref{l:up}. We will produce a similar ascending chain as in Case 3, however this time, we have to pick the modules from different classes $\mathcal C_\delta$, hence we lose continuity. To overcome this problem, we use a well known singular compactness argument:

We gradually build the chains $\mathcal N_i = (N^i_\alpha\mid\alpha<\mu)$, for $i<\omega$, and pick the modules $N_i\in\mathcal C_\lambda$ in an alternating way, so that the following conditions are satisfied for all $i<\omega$:
\begin{enumerate}
\item[$(a)$] $N^i_\alpha\in\mathcal C_{\nu ^+_\alpha}$;
\item[$(b)$] $\bigcup \mathcal N_i\subseteq N_{i+1}$;
\item[$(c)$] the generators $\{n_{\alpha,\beta}^i\mid \beta<\nu_\alpha\}$ of the modules $N^i_\alpha$ are fixed;
\item[$(d)$] the generators $\{n_\gamma^i \mid \gamma<\nu\}$ of $N_i$ are fixed;
\item[$(e)$] $N^i_\alpha\supseteq\{n^{i-1}_{\delta,\beta} \mid \delta<\mu\;\&\; \beta<\hbox{min}(\nu _\delta, \nu _\alpha)\}\cup\{n_\gamma^i \mid \gamma<\nu _\alpha\}$.
\end{enumerate}

Then $\bigcup_{i<\omega}N_i$ is in $\mathcal C_\lambda$, and it is equal to the union of the chain $\mathcal H = (\bigcup _{i<\omega} N^i_\alpha \mid \alpha<\mu)$. However, this chain is continuous (by the condition $(e)$), and provides thus an $\mathcal A$-filtration of $\bigcup\mathcal H$.
\smallskip

Having constructed the families $\mathcal C_\lambda$ ($\lambda \leq \kappa$), we can use them to build an $\mathcal A$-filtration of $M$ consisting of $<\kappa$-presented modules. If $\kappa$ is regular, then it is easy to see that $\mathcal C_\kappa$ already contains an $\mathcal A$-filtration of $M$. For $\kappa$ singular, we apply \cite[Lemma 3.2]{SS}. By inductive hypothesis, we conclude that $M$ is $\mathcal A_0$-filtered.

It remains to show that the modules in $\mathcal A_0$ have countably presented first syzygies. We know from Lemma~\ref{l:kerepi} that their first syzygies are strict $\mathcal B$-stationary. In particular, the first syzygies are ${}_RR^c$-stationary. By \cite[Proposition 8.14(1)]{AH}, they are then ${}_RR$-Mittag-Leffler, and the claim follows from \cite[Corollary 5.3]{AH}. Finally, $\mathcal B$ is definable by \cite[Theorem 13.41]{GT}.
\end{proof}

We finish this section by examples of cotorsion pairs which satisfy the conditions of Theorem \ref{t:counttype}, but are not hereditary. 

Recall that for $n < \omega$, a module $M$ is an \emph{FP$_n$-module} provided that $M$ has a projective resolution 
$$\dots \to P_{k+1} \to P_k \to \dots \to P_1 \to P_0 \to M \to 0$$ 
such that $P_k$ is finitely generated for each $k \leq n$, \cite[\S VIII.4]{Br}. Let $\mathcal F \mathcal P _n$ denote the class of all FP$_n$-modules. (Notice that $\mathcal F \mathcal P _0$ is the class of all finitely generated modules, and $\mathcal F \mathcal P _1$ the class of all finitely presented ones.) The classes $( \mathcal F \mathcal P _n \mid n < \omega )$ form a decreasing chain whose intersection is the class of all strongly finitely presented modules, 
that is, the modules possessing a projective resolution consisting of finitely generated modules, see \cite[Proposition VII.4.5]{Br}.    

\medskip     
\begin{exm}\label{e:nonher} Let $n \geq 2$ and let $R_n$ be a ring such that there exists a module $M \in \mathcal F \mathcal P _n \setminus \mathcal F \mathcal P _{n+1}$. Such rings were constructed by Bieri and Stuhler using integral representation theory, see \cite[\S VIII.5]{Br} and the references therein; in fact, $R_n$ can be taken as the integral group ring $\mathbb{Z} G_n$ for a suitable finitely generated group $G_n$, and $M = \mathbb Z$. (For the particular case of $n = 2$, a different and simpler example of $R_2$ is constructed in \cite[Example 1.4]{KM}.)    

Let $\mathfrak C _n = (\mathcal A _n, \mathcal B _n)$ be the cotorsion pair generated by $\mathcal F \mathcal P _n$. Since $n \geq 2$, $\mathcal B _n$ is closed under direct limits by \cite[Lemma 6.6]{GT}, so Theorem \ref{t:counttype} applies. 

In order to see that $\mathfrak C _n$ is not hereditary, we observe that $\mathcal A _n$ is not closed under kernels of epimorphisms: indeed, the syzygy module $\Omega (M)$ does not belong to $\mathcal F \mathcal P _n$, and hence $\Omega (M) \notin \mathcal A _n$. To see the latter fact, notice that $\mathcal A _n$ consists of direct summands of $\mathcal F \mathcal P _n$-filtered modules by \cite[6.14]{GT}. Since $\Omega (M)$ is finitely generated, if $\Omega (M) \in \mathcal A _n$ then it is a direct summand of a finitely $\mathcal F \mathcal P _n$-filtered module, by \cite[Theorem 7.10]{GT} used for $\aleph_0$. However, $\mathcal F \mathcal P _n$ is closed under extensions and direct summands, whence $\Omega(M) \in \mathcal F \mathcal P _n$, a contradiction.    
\end{exm}


\end{document}